\documentclass[12pt]{amsart}                     

\usepackage{amssymb}
\usepackage[all]{xy}

\newtheorem{thm}{Theorem}[section]
\newtheorem{prop}[thm]{Proposition}
\newtheorem{cor}[thm]{Corollary}
\newtheorem{lem}[thm]{Lemma}
\theoremstyle{definition}
\newtheorem{defin}[thm]{Definition}
\theoremstyle{remark}
\newtheorem{remk}[thm]{Remark}

\newcommand{\gM}{\mathfrak{m}}		\newcommand{\gR}{\mathfrak{r}}
\newcommand{\al}{\alpha}					\newcommand{\be}{\beta}
\newcommand{\la}{\lambda}					\newcommand{\ga}{\gamma}
\newcommand{\Om}{\Omega}				\newcommand{\de}{\delta}
\newcommand{\cQ}{\mathcal{Q}}			\newcommand{\kT}{\mathcal{T}}
\newcommand{\mP}{\mathbb{P}}
\newcommand{\bP}{\mathbf{P}}		

\newcommand{\oK}{\bar{K}}				\newcommand{\oP}{\bar{P}}	
\newcommand{\obP}{\bar{\mathbf{P}}}		\newcommand{\tF}{\tilde{F}}
\newcommand{\hP}{\hat{P}}		\newcommand{\hbP}{\hat{\bP}}
\newcommand{\hH}{\hat{H}}		\newcommand{\hxi}{\hat{\xi}}	
\newcommand{\hu}{\hat{u}}		\newcommand{\hv}{\hat{v}}

\newcommand{\<}{\langle}   	\renewcommand{\>}{\rangle}
\renewcommand{\*}{\otimes}   \newcommand{\+}{\oplus}
\newcommand{\bop}{\bigoplus}
		\newcommand{\8}{\infty}
\newcommand{\bu}{{ \scriptstyle \bullet}}

\newcommand{\sh}{^\sharp}			
\newcommand{\ch}{^{\scriptscriptstyle\vee}}
\newcommand{\spe}{\supset}		\newcommand{\sbe}{\subset}
\newcommand{\pmd}{\hskip-2pt\pmod{p}}
\newcommand{\ppmd}{\hskip-2pt\pmod{p^2}}

\newcommand{\lat}{\mbox{-}\mathrm{lat}}
\newcommand{\Hom}{\mathop\mathrm{Hom}\nolimits}
\newcommand{\Ker}{\mathop\mathrm{Ker}\nolimits}

\newcommand{\Tor}{\mathop\mathrm{Tor}\nolimits}
\newcommand{\Ext}{\mathop\mathrm{Ext}\nolimits}
\newcommand{\rad}{\mathop\mathrm{rad}\nolimits}
\newcommand{\rep}{\mathop\mathrm{rep}\nolimits}
\newcommand{\rk}{\mathop\mathsf{Rk}\nolimits}

\newcommand{\AR}{Auslander-Reiten}

\begin{document}

\title[Kleinian 4-rings]{Representations and cohomologies of Kleinian 4-rings}

\author{Yuriy Drozd}
\thanks{This paper was prepared during the stay of the author at the Max-Plank-Institute for Mathematics (Bonn).}

\address{Institute of Mathematics NAS of Ukraine, Kyiv, Ukraine}
 \email{y.a.drozd@gmail.com}             
\urladdr{www.imath.kiev.ua/~drozd}



\maketitle

\begin{abstract}
 We introduce a new class of algebras over discrete valuation rings, called \emph{Kleinian 4-rings}, 
 which generalize the group algebra of the Kleinian 4-group. For these algebras we describe the lattices
 and their cohomologies. In the case of \emph{regular lattices} we obtain an explicit form of cocycles
 defining the cohomology classes.
\end{abstract}


\section*{Introduction}
\label{intro}

Integral representations of the Kleinian 4-group $G$ (or $G$-lattices) were described by Nazarova \cite{naz22}. Another description was
proposed by Plakosh \cite{plak1}. In the papers \cite{dpl2} and \cite{regular} cohomologies of these lattices were calculated. In this paper
we consider a class of rings that generalizes group rings of the Kleinian 4-group. We call them \emph{Kleinian 4-rings}. We give a description
of lattices over such rings and calculate cohomologies of these lattices. In a special case of \emph{regular lattices} we obtain an explicit
form of cocycles defining cohomology classes. 
 
\section{Lattices over Kleinian 4-rings}
\label{sec1}

 In what follows $R$ denotes a complete discrete valuation ring with a prime element $p$, the field of fractions $Q$ and the field of residues
 $\Bbbk=R/pR$. We write $\*$ instead of $\*_R$. If $A$ is an $R$-algebra, we call an $A$-module $M$ an \emph{$A$-lattice} if it is finitely
 generated and free as $R$-module. Then we identify $M$ with its image $1\*M$ in the vector space $Q\*M$ and an element $v\in M$ with
 $1\*v\in Q\*M$. We denote by $A\lat$ the category of $A$-lattices.

 \begin{defin}
  The \emph{Kleinian 4-ring} over $R$ is the $R$-algebra $K=R[x,y]/(x(x-p),y(y-p))$.
 \end{defin}
 
 Note that if $p=2$ this is just the group algebra over $R$ of the \emph{Kleinian 4-group}  $G=\< a,b \mid a^2=b^2=1,\,ab=ba \>$.
 One has to set $x=a+1,\,y=b+1$. 
 
 One easily sees that $Q\*K$ is isomorphic to $Q^4$: just map $x$ to $\bar x=(p,p,0,0)$ and $y$ to $\bar{y}=(p,0,p,0)$. We consider $K$
 as embedded into $Q^4$ identifying $x$ with $\bar{x}$ and $y$ with $\bar{y}$. We also set $z=(p,0,0,0)\in Q^4$ (note that $z\notin K$ and
 $z^2=xy)$. The maximal ideal $\gR$ of $K$ is $(p,x,y)$ and $K/\gR\simeq\Bbbk$. Let $A=\{a\in Q^4\mid a\gR\sbe K\}$. One easily verifies that
 $A=K+Rz$ and $A/K\simeq\Bbbk$.  Hence $K$ is a \emph{Gorenstein ring} \cite[Proposition~6]{ideals}, i.e.
 $\mathrm{inj.dim}_KK=1$. Therefore, $A$ is its unique minimal over-ring and every $K$-lattice is isomorphic to a direct sum of a free $K$-module
 and an $A$-lattice (see \cite[Lemma~2.9]{qbass} or \cite[Lemma~3.2]{rejection}). Note that the ring $A$ is also local with the maximal ideal
 $\gM=(p,x,y,z)$ and $A/\gM\simeq\Bbbk$. Moreover, as the submodule of $Q^4$, $\gM=pA\sh=\rad A\sh$, where $A\sh=R^4$ is hereditary. 
 Thus $A$ is a \emph{Backstr\"om order} in the sense of \cite{rinrog}. Therefore, $A$-lattices can be described by the representations of the quiver
 \[
    \Gamma=  \vcenter
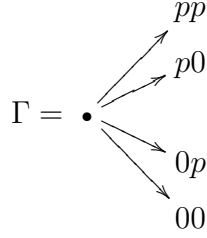
{\xymatrix@R=1ex{ & {pp} \\ & {p0} \\ \bu \ar[uur] \ar[ur] \ar[dr] \ar[ddr] \\ &{0p} \\ &{00}	  }  }
\]  
 over the field $\Bbbk$.
 Namely, denote by $R_{\al\be}$, where $\al,\be\in\{0,p\}$ the $A$-lattice such that $R_{\al\be}=R$ as $R$-module, $xv=\al v$ and $yv=\be v$
 for all $v\in R_{\al\be}$. For any $A$-lattice $M$ and $\al,\be\in\{0,p\}$ set $M_{\al\be}=\{v\in M\mid xv=\al v,\,yv=\be v\}$. If $M$ is an $A$-lattice, 
 $M\sh=A\sh M$ is an $A\sh$-module, hence $M\sh=\bop_{\al,\be}M\sh_{\al\be}$. Let $V_\bu=M/\gM M$ and $V_{\al\be}=M\sh_{\al\be}/pM\sh_{\al\be}$.
 Note that $M\sh\spe M\spe \gM M= pM\sh$. So the natural maps $f_{\al\be}:V_\bu\to V_{\al\be}$ are defined and we obtain a representation $V$ of
 the quiver $\Gamma$:
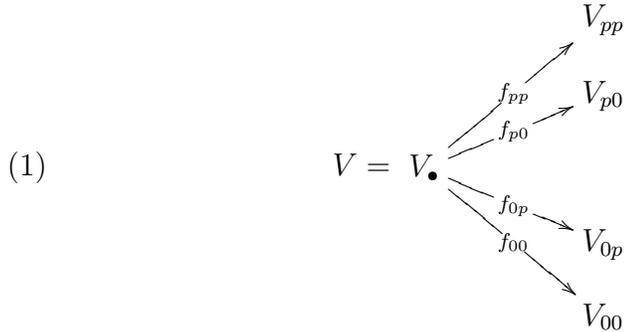
   \begin{equation}\label{e11}  
     V=  \vcenter{\xymatrix@R=2ex@C=4em{ & V_{pp} \\ & V_{p0} \\ 
    V_\bu \ar[uur]|{\,f_{pp}\,} \ar[ur]|{\,f_{p0}\,} \ar[dr]|{\,f_{0p}\,} \ar[ddr]|{\,f_{00}\,} \\ & V_{0p} \\ & V_{00}	  }  }
 \end{equation}
 We denote this representation by $\Phi(M)$. It gives a functor $\Phi:A\lat\to\rep\Gamma$.
  The next result follows from \cite{rinrog}.
 
 \begin{thm}\label{t11} 
  Let $\rep_+\Gamma$ be the full subcategory of $\rep\Gamma$ consisting of such representations $V$ that all maps $f_{\al\be}$ are surjective and the map
  $f_+:V_\bu\to V_+$ is injective. The image of the functor $\Phi$ is in $\rep_+\Gamma$ and, considered as the functor $A\lat\to\rep_+\Gamma$, the functor
  $\Phi$ is an \emph{epivalence}.
 \end{thm}
 
 Recall that the term \emph{epivalence} means that $\Phi$ is full, maps non-isomorphic objects to non-isomorphic
and every representation $V\in\rep_+\Gamma$ is isomorphic to some $\Phi(M)$ (then $\Phi$ maps indecomposable objects to indecomposable).  
 Actually, this $M$ can be reconstructed as follows.
 Set $d_{\al\be}=\dim V_{\al\be}$, $V_+=\bop_{\al\be}V_{\al\be}$ and $\bar{V}$ be the image of the map $V_\bu\to V_+$ with the components $f_{\al\be}$. 
 Then $V_+\simeq M\sh/pM\sh$, where $M\sh=\bop_{\al\be}R_{\al\be}^{d_{\al\be}}$. Let $\Psi(V)$ be the preimage of $V_+$ in $M\sh$. It is an $A$-lattice and 
 $\Phi(M)\simeq V$. Moreover, $M\sh=A\sh M$. Note also that the kernel of the map $\Hom_A(M,N)\to\Hom_\Gamma(\Phi(M),\Phi(N))$ coincides with
 $\Hom_A(M,\gM N)$.
  
 The quintuple $(d_\bu\mid d_{pp},d_{p0},d_{0p},d_{00})$, where $d_\bu=\dim V_\bu$, is called the \emph{vector dimension} of the representation $V$. We
 also call it the \emph{vector rank} of the lattice $M=\Psi(V)$ and denote it by $\rk M$.  For instance, $\rk R_{pp}=(1\mid1,0,0,0)$ and $\rk A=(1\mid1,1,1,1)$.
 Note that the rank of $M$ as of $R$-module equals $\sum_{\al\be}d_{\al\be}$, while $d_\bu=\dim_\Bbbk M/\gM M$.

 \begin{remk}\label{r12} 
  Note that the only indecomposable representations of $\Gamma$ that do not belong to $\rep_+\Gamma$ are ``\emph{trivial representations}'' $V^j$, where
  $j\in\{\bu,\,\al\be \mid \al,\be\in\{0,p\}\}$ such that $V^j_j=\Bbbk$ and $V^j_{j'}=0$ if $i\ne j$. Therefore, $A$-lattices are indeed classified by 
  representations of the quiver $\Gamma$.
 \end{remk}
  
 Let $\tau_K$ ($\tau_A$) denote the \emph{\AR translate} in the category $K\lat$ (respectively, $A\lat$). Recall that $\tau_KM$ for a non-projective indecomposable
 $K$-lattice $M$ is an indecomposable $K$-lattice $N$ such that there is an \emph{almost split sequence} $0\to N\to E\to M\to0$ \cite{AusIsolated}. The next result
 follows from \cite{rejection}.
 \begin{prop}\label{p13} 
  \begin{enumerate}
  \item    $\tau_KM\simeq\tau_AM$ for any indecomposable $A$-lattice $M\not\simeq A$.
  \item    $\tau_KA\simeq\gR$ and it is a unique indecomposable $A$-lattice $N$ such that $\mathrm{inj.dim}_N=1$.
  \item  $\tau_KM\simeq\Omega M$ for any $A$-lattice $M$, where $\Omega M$ denote the \emph{syzygy} of $M$
  as of $K$-module.
  \end{enumerate}  
 \end{prop}
 
  Following \cite{rogback}, we can also restore the \AR\ quiver $\cQ(A)$ of the category $A\lat$ from the \AR\ quiver $\cQ(\Gamma)$r of the category $\rep\Gamma$. 
  Recall that the quiver $\cQ(\Gamma)$ consists of the \emph{preprojective, preinjective} and \emph{regular} components. The quiver $\cQ((A)$ is obtained from
  $\cQ(\Gamma)$ by glueing the preprojective and preinjective components omitting trivial representations. The resulting \emph{preprojective-preinjective}
  component is the following:
  {\footnotesize
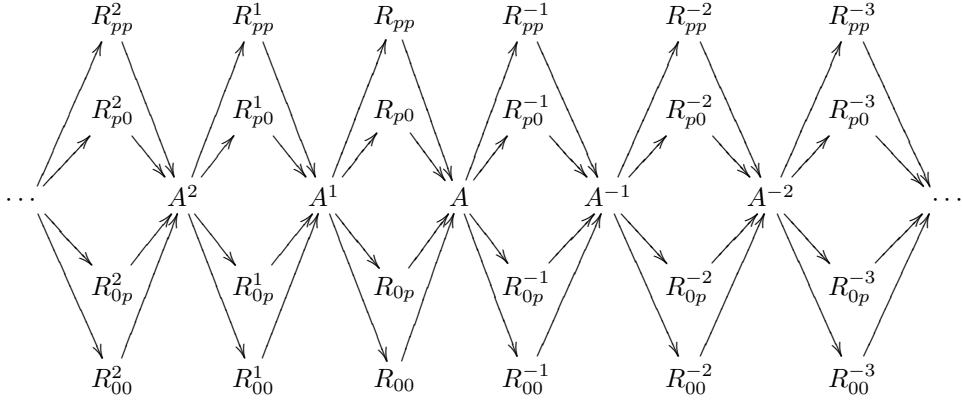
 \[
  \xymatrix@C=.65em@R=1.7em{ & R_{pp}^2 \ar[ddr] && R_{pp}^1 \ar[ddr] && R_{pp} \ar[ddr] && R_{pp}^{-1} \ar[ddr] && R_{pp}^{-2} \ar[ddr]
  && R_{pp}^{-3} \ar[ddr] \\
  & R_{p0}^2 \ar[dr] && R_{p0}^1 \ar[dr] && R_{p0} \ar[dr] && R_{p0}^{-1} \ar[dr] && R_{p0}^{-2} \ar[dr]
  && R_{p0}^{-3} \ar[dr] \\
 {\cdots\ \ {}} \ar[uur] \ar[ur] \ar[ddr] \ar[dr] && A^2 \ar[uur] \ar[ur] \ar[ddr] \ar[dr] && A^1 \ar[uur] \ar[ur] \ar[ddr] \ar[dr]  && A \ar[uur] \ar[ur] \ar[ddr] \ar[dr] 
  && A^{-1} \ar[uur] \ar[ur] \ar[ddr] \ar[dr] && A^{-2} \ar[uur] \ar[ur] \ar[ddr] \ar[dr] && {{}\ \ \cdots}\\
  & R_{0p}^2 \ar[ur] && R_{0p}^1 \ar[ur]  && R_{0p} \ar[ur] && R_{0p}^{-1} \ar[ur] && R_{0p}^{-2} \ar[ur]
  && R_{0p}^{-3} \ar[ur] \\
  & R_{00}^2 \ar[uur] && R_{00}^1 \ar[uur]  && R_{00} \ar[uur] && R_{00}^{-1} \ar[uur] && R_{00}^{-2} \ar[uur]
  && R_{00}^{-3} \ar[uur]
   }
 \]
 }
 Here $M^k$ denotes $\tau_K^kM$. Note that $A^1\simeq\gR\simeq A\ch$, where $M\ch=\Hom_K(M,K)$. The representations belonging to this component are
 uniquely determined by their vector-ranks. One can verify that
 \begin{align*}
  \rk A^k&=
  \begin{cases}
   (2k-1\mid k,k,k,k) &\text{if } k>0,\\
   (1-2k\mid 1-k,1-k,1-k,1-k) &\text{if } k<0;
  \end{cases} \\
  \rk R_{pp}^k&=
  \begin{cases}
   (k+1\mid \left[\frac{k}{2}\right]-(-1)^k,\left[\frac{k}{2}\right],\left[\frac{k}{2}\right],\left[\frac{k}{2}\right]) &\text{if } k>0,\\
   (-k\mid \left[\frac{1-k}{2}\right]+(-1)^k,\left[\frac{1-k}{2}\right],\left[\frac{1-k}{2}\right],\left[\frac{1-k}{2}\right]) &\text{if } k<0.
  \end{cases}
 \end{align*} 
 $\rk R_{\al\be}^k$ is obtained from $\rk R_{pp}^k$ by permutation of $d_{pp}$ with $d_{\al\be}$.
 
 The remaining (regular) components are \emph{tubes}, where $\tau_K$ acts periodically. They are parametrized by the set
 $$\mP=\{\mbox{irreducible unital polynomials } f(t)\in\Bbbk[t]\}\cup\{\8\}.$$
 Actually, it is the set of closed points of the  projective line over the field $\Bbbk$, that is of the projective scheme $\mathrm{Proj\,}\Bbbk[x,y]$.
 If $f(t)=t-\la$ ($\la\in\Bbbk$), we write $\kT^\la$ instead of $\kT^f$. 
 
 If $f\in\mP\setminus\{t,t-1,\8\}$, the corresponding tube $\kT^f$ is \emph{homogeneous}, which means that $\tau_K M\simeq M$ for all $M\in\kT^f$.
 It has the form
 \[
   \xymatrix{ T^f_1 \ar@/^/[r]	 & T^f_2 \ar@/^/[r] \ar@/^/[l] & T^f_3 \ar@/^/[r] \ar@/^/[l] & \cdots  \ar@/^/[l]} 
 \]
 and  $\rk T^f_n=(2dn\mid dn,dn,dn,dn)$, where $d=\deg f(t)$. In this diagram all maps $T^f_n\to T^f_{n+1}$ are monomorphisms with the cokernels
 $T^f_1$, while all maps $T^f_{n+1}\to T^f_n$ are epimorphisms with the kernels $T^f_1$.
 
 The \emph{exceptional tubes} $\kT^\la\ (\la\in\{0,1,\8\})$ are of he form 
  \begin{equation}\label{stube} 
  \vcenter{ \xymatrix{ T^{\la1}_1 \ar[r]  & T^{\la1}_2 \ar[r] \ar[dl] & T^{\la1}_3 \ar[r] \ar[dl]& T^{\la1}_4 \ar[r] \ar[dl] \ar[r] \ar[dl] & \cdots \ar[dl]\\
   T^{\la2}_1 \ar[r]  & T^{\la2}_2 \ar[r] \ar[ul] & T^{\la2}_3 \ar[r] \ar[ul]& T^{\la2}_4 \ar[r] \ar[ul] \ar[r] \ar[ul] & \cdots \ar[ul]} }
 \end{equation}
  Here  $\tau T^{\la1}_n=T^{\la2}_n$ and $\tau T^{\la2}_n=T^{\la1}_n$. In this diagram all maps $T^{\la i}_n\to T^{\la i}_{n+1}$ are monomorphisms
  with the cokernels $T^{\la j}_1$, where $j=i$ if $n$ is even and $j\ne i$ if $n$ is odd. All maps $T^{\la i}_{n+1}\to T^{\la j}\ (j\ne i)$ are epimorphisms
  with the kernels $T^{\la i}_1$. 
  
  For $\la=1$ we have
  \begin{align}\label{erk} 
  & \rk T^{1j}_{2m}=(2m\mid m,m,m,m) \ \text{ for both $j=1$ and $j=2$}, \notag\\
   & \rk T^{11}_{2m-1}=(2m-1\mid m,m,{m-1},{m-1}),\\
   & \rk T^{12}_{2m-1}=(2m-1\mid m-1,m-1,mm).  \notag 
  \end{align}
 The vector-ranks for the tubes $\kT^0$ and $\kT^\8$ are obtained from those for $\kT^1$ by permutation of $d_{p0}$, respectively, with $d_{00}$ and 
 with $d_{0p}$. 
 
 \section{Cohomologies}
 \label{sec2} 
 
 A Kleinian 4-ring is a \emph{supplemented $R$-algebra} in the sense of \cite[Ch.\,X]{CE} with respect to the surjection $\pi:K\to K/(x-p,y-p)\simeq R$. 
 Therefore, for any $K$-module $M$ the homologies $H_n(K,M)=\Tor_N^K(R,M)$ and cohomologies $H^n(K,M)=\Ext^n_K(R,M)$ are defined.
 Moreover, if we consider $M$ as $K$-bimodule setting $mx=my=pm$ for all $m\in M$, they coincide with the Hochschild homologies and
 cohomologies:
 \[
    H_n(K,M)\simeq HH_n(K,M)\ \text{ and } \ H^n(K,M)\simeq HH^n(K,M).
 \]
 (see \cite[Theorem~X.2.1]{CE}). 
 
\begin{remk}\label{r21} 
  We have chosen the augmentation $K\to R$ such that if $p=2$, hence $K\simeq RG$ for the Kleinian 4-group $G$, it coincides with the usual
  augmentation $RG\to R$ mapping all elements of the group to $1$. Thus in this case $H_n(K,M)=H_n(G,M)$.
 \end{remk} 
 
 \begin{prop}\label{p22}
  For every $K$-module $M$ and $n\ne0$
  \[
   xyH_n(K,M)=xyH^n(K,M)=p^2H^n(K,M)=p^2H_n(K,M)=0  
  \]
 \end{prop}
 \begin{proof}
  The map $\mu:r\mapsto rxy$ is a homomorphism of $K$-modules $R\to K$ such that $\pi\mu:R\to R$ is the multiplication by $xy$ or, the same, by $p^2$.
  Therefore, the multiplication by $xy$ or by $p^2$ in $\Ext^n_K(R,M)$ or in $\Tor_n^K(R,M)$ factors, respectively, through $\Ext^n_K(K,M)=0$ or
  through $\Tor^K_n(K,M)=0$.
 \end{proof}
  
  Note that $K\simeq \oK\*\oK$, where $\oK=R[x]/(x(x-p))$
  A projective resolution $\obP$ for $R$ as of $\oK$-module, where $xr=pr$ for all $r\in R$, is obtained if we set $\oP_n=\oK u^n$ and
 \[
    du^n=C_n(x)u^{n-1},\
	\text{where}\ C_i(x)=
  \begin{cases}
   x &\text{if $n$ is even},\\
   x-p &\text{if $n$ is odd}.
  \end{cases}
  \]
 Then $\bP=\obP\*\obP$ is a projective resolution of $R$ as of $K$-module. Here $P_n$ is the module of homogeneous polynomials of degree
 $n$ from $K[u,v]$ and
 \[
  d(x^iy^j)=C_i(x)u^{i-1}v^j+(-1)^iC_j(y)u^iv^{j-1}.
 \]
 
 Denote $H_n(\oK,M)=\Tor^{\oK}_n(R,M)$. Then
 \[
  H_n(\oK,M)=
  \begin{cases}
   M/(x-p)M &\text{if } n=0,\\
   \Ker(x-p)_M/xM &\text{if  $n$ is odd},\\
   \Ker x_M/(x-p)M &\text{if $n$ is even},
  \end{cases}
 \]
 where $a_M$ denotes the multiplication by $a$ in the module $M$.
 Let $R_0=\oK/(x),\,R_p=\oK/(x-p)$. Then $R_{\al\be}\simeq R_\al\*R_\be$. As the ring $R$ is hereditary, the K\"unneth formula \cite[Theorem~VI.3.2]{CE} implies that
 \begin{align*}
  H_n(K,R_{\al\be})\simeq &   \textstyle\big(\bop_{i+j=n}H_i(\oK,R_\al)\*H_j(\oK,R_\be)\big)\+\\
  \+&   \textstyle\big(\bop_{i+j=n-1} \Tor_1^R(H_i(\oK,R_\al),H_j(\oK,R_\be))\big).
 \end{align*}
 Since
 \begin{align}
  H_n(\oK,R_0)&=
  \begin{cases}
   0 &\text{if $n$ is odd},\\
   \Bbbk &\text{if $n$ is even};
  \end{cases}	\notag\\
   H_n(\oK,R_p)&=
  \begin{cases}
   R &\text{if } n=0,\\
   \Bbbk &\text{if $n$ is odd},\\
   0 &\text{if $n$ is even},
  \end{cases}	\notag\\
  \intertext{we obtain}
  H_n(K,R_{pp})&=
  \begin{cases}
   R &\text{if } n=0,\\
   (R/p)^{[(n+3)/2]} &\text{if $n$ is odd},\\
  (R/p)^{n/2} &\text{if $n$ is even};
  \end{cases} \label{e21} \\
  \intertext{and if $(\al,\be)\ne(p,p)$}
    H_n(K,R_{\al\be})&=(R/p)^{[(n+2)/2]}. \label{e22} 
 \end{align}
  
  On the other hand, the exact sequence $0\to K\to A\to \Bbbk\to0$ implies that for $n>0$  
  \begin{equation}\label{e23} 
  H_n(K,A)\simeq H_n(K,\Bbbk)\simeq P_n\*_K\Bbbk\simeq \Bbbk^{n+1},
  \end{equation}
   since $H_n(K,K)=0$ and the differential in $\bP\*_K\Bbbk$ is zero.
   
  As $K$ is Gorenstein, the functor $M\mapsto M\ch=\Hom_K(M,K)$ is an exact duality in the category $K\lat$, i.e. the natural map
  $M\mapsto M^{\scriptscriptstyle{\vee\vee}}$ is an isomorphism. If $P$ is projective, then $P\*_KM\simeq\Hom_K(P\ch,M)$. Therefore, homologies
  of a module $M$ can be obtained as $H_n(\Hom_K(\bP\ch,M)$. Note that the embedding $R\to P_0\ch\simeq K$ maps $1$ to $xy$.
  Hence, just as for finite groups, we can consider a \emph{full resolution} $\hbP$ setting 
  \[
    \hP_n=
    \begin{cases}
     P_n &\text{if } n\ge 0,\\
     P\ch_{-n-1} &\text{if } n<0
    \end{cases}  
  \]
  and defining $d_0:K=\hP_0\to\hP_{-1}\simeq K$ as multiplication by $xy$. Thus the \emph{Tate cohomologies} $\hH^n(K,M)$ are
  defined as $H^n(\Hom_K(\hbP,M))$ with the usual properties
  \[
   \hH^n(K,M)=
   \begin{cases}
    H^n(K,M) &\text{if } n>0,\\
    H_{-n-1}(K,M) &\text{if } n<-1,\\
    M_{pp}/xyM &\text{if } n=0,\\
    \{m\mid xym=0\}/((x-p)M+(y-p)M) &\text{if } n=-1,
   \end{cases}
  \]
  where $M_{pp}=\{m\mid xm=ym=pm\}$.
  In particular, $xy\hH^n(K,M)=p^2\hH^n(K,M)=0$ for all $M$. Note also that, if $M$ is an $A$-lattice, 
  $M_{pp}=\{m\mid zm=pm\}$ and $xyM=z^2M$.
  
  A basis of $\hP_{-n}\ (n>0)$ can be chosen as $\{\hu^i\hv^j\mid i+j=n-1\}$, where $(\hu^i\hv^j)(u^kv^l)=\de_{ik}\de_{jl}$. Then
  \[
   d(\hu^i\hv^j)=C_{i+1}\hu^{i+1}\hv^j+(-1)^iC_{j+1}\hu^i\hv^{j+1}.
  \]
  
  \begin{prop}\label{p23} 
   If $M$ is an $A$-lattice that has no direct summands isomorphic to $R_{pp}$, then $$\hH^0(K,M)= M_{pp}/pM_{pp}\simeq \Bbbk^{d_{pp}},$$
   where $(d_\bu\mid d_{pp},d_{p0},d_{0p},d_{00})=\rk M$.
  \end{prop}
  \begin{proof}
   Set $M\sh=A\sh M=\bop_{\al\be}M\sh_{\al\be}$. Note that $xyA=Rxy=xyA\sh$, hence $xyM=xyM\sh=p^2M\sh_{pp}$. On the other hand, 
   $M\sh_{pp}\simeq R_{pp}^{d_{pp}}$ and $pM\sh_{pp}\sbe M_{pp}\sbe M\sh_{pp}$. If $M_{pp}\ne pM\sh_{pp}$, $M_{pp}$ contains a direct summand 
   $L\simeq R_{pp}$ of $M\sh_{pp}$. Then $M\sh=L\+L'$ and $M=L\+(L'\cap M)$, which is impossible. Therefore, $M_{pp}=pM\sh_{pp}$,
   $xyM=pM_{pp}$ and $\hH^0(K,M)= M_{pp}/pM_{pp}\simeq\Bbbk^{d_{pp}}$.
  \end{proof}
  
  Denote $T=Q/R$, $DM=\Hom_R(M,T)$. It is the Matlis duality between noetherian and artinian $R$-modules, as well as $K$-modules 
  \cite{matlis}. We have the following dualities for cohomologies.
  
  \begin{prop}\label{p24} 
   Let $M$ be a $K$-module. Then
   \vspace*{-1ex}
   \begin{align}
 & \hH^n(K,DM)\simeq D\hH^{-n-1}(K,M), \label{du1}\\
  \intertext{and if $M$ is a lattice}
 & \hH^n(K,DM)\simeq \hH^{n+1}(K,M\ch), \label{du2}\\
& \hH^n(K,M\ch)\simeq D\hH^{-n}(K,M).  \label{du3} 
   \end{align}
  \end{prop}
  \begin{proof}
  Note first that, since $K$ is local and Gorenstein, $\Hom_R(K,R)\simeq K$, whence $M\ch\simeq \Hom_R(M,R)$ and we identify these modules.
  As $T$ is an injective $R$-module, 
  \[
     \Ext^n_K(R,\Hom_R(M,T))\simeq \Hom_R(\Tor_n^K(R,M),T),
  \]
  (see \cite[Proposition~VI.5.1]{CE}), which is just \eqref{du1}.
  
  The exact sequence $0\to R\to Q\to T\to0$ gives, for any lattice $M$, the exact sequence
  \[
    0\to M\ch\to \Hom_R(M,Q)\to DM \to 0.  
  \]
  As multiplication by $p^2$ is an automorphism of $\Hom_R(M,Q)$, Proposition~\ref{p22} implies that $\hH^n(\Hom_R(M,Q))=0$. Then the 
  long exact sequence for cohomologies implies \eqref{du2}.
  
  \eqref{du3} is a combination of \eqref{du1} and \eqref{du2}.
  \end{proof}
  
  Note also that $\hH^n(K,F)=0$ for any projective (hence free) $K$-module $F$. Therefore, Proposition~\ref{p13} implies that, for any
  indecomposable $A$-lattice $M$,
  \begin{equation}\label{e24} 
   \hH^n(K,M)\simeq \hH^{n+1}(K,\tau_K M)\simeq \hH^{n-1}(K,\tau^{-1}_KM).
  \end{equation}  
  
  Hence from the formulae \eqref{e21}-\eqref{e23} and the duality \eqref{du3} we obtain a complete description of cohomologies of $K$-lattices 
  belonging to the preprojective-preinjective component.
  
  \begin{thm}  \label{t25} 
 \begin{align*}    
   & \hH^n(K,A^k)\simeq
   \begin{cases}
    \Bbbk^{n-k+1} &\text{ \rm if } n\ge k,\\
    \Bbbk^{k-n} &\text{ \rm if } n<k;
   \end{cases}\\
   & \hH^n(K,R_{pp}^k)\simeq
   \begin{cases}
    \Bbbk^{(|n-k|/2+1)} &\text{ \rm if } n-k\ne 0 \text{ \rm is even},\\
    \Bbbk^{|n-k|/2]} &\text{ \rm if } n-k \text{ \rm is odd},\\
    R/xyR &\text{ \rm if } n=k;
   \end{cases}\\
& \hH^n(K,R_{\al\be}^k)\simeq  \Bbbk^{[(|n-k|+1)/2]}\ \text{ \rm if } (\al,\be)\ne(p,p).
  \end{align*}
  \end{thm}
   
   The description of cohomologies of $A$-lattices belonging to tubes are obtained from Proposition~\ref{p23}, since $\hH^n(K,M)\simeq \hH^0(K,\tau^{-n}_KM)$
   and we know the action of $\tau_K$ in tubes.
   
   \begin{thm}\label{t26}     
    \begin{enumerate}
    \item If $f\notin\{t,t-1\}$, then $\hH^n(K,T^f_m)\simeq\Bbbk^{dm}$, where $d=\deg f$.
    \smallskip
    \item If $\la\in\{0,1,\8\}$, then
    \[
    \hH^n(K,T_m^{\la i})\simeq
    \begin{cases}
     \Bbbk^{m/2} &\text{ \rm if $m$ is even},\\
     \Bbbk^{(m-(-1)^{n+i})/2} &\text{ \rm if $m$ is odd}.
    \end{cases}    
    \]
    \end{enumerate}
   \end{thm}
  
  \section{Regular lattices}
  \label{sec3} 
 
  An $A$-lattice $M$ is called \emph{regular} if all its indecomposable direct summands belong to tubes. As neither regular lattice is projective,
  $\tau_KM=\tau_AM=\Om M$.  Note that if $M$ is regular, then
  \begin{align}
  & 2d_\bu(M)={\sum}_{\al\be} d_{\al\be}(M). \label{e25}\\
  \intertext{Therefore,}
  &  d_\bu(\Om M)=d_\bu(M) \ \text{ and }\ d_{\al\be}(\Om M)=d_\bu(M)-d_{\al\be}(M). \label{e26} 
  \end{align}
 These formulae imply the following fact.
   
  \begin{lem}\label{l31} 
   Every exact sequence of regular $A$-lattices $0\to M\to N\to L\to0$ induces exact sequences
   \begin{align}
    & 0\to\Om M\to\Om N\to\Om L\to0,  \label{e27} \\
    & 0\to\Om^{-1} M\to\Om^{-1} N\to\Om^{-1} L\to0. \label{e28}
   \end{align}
  \end{lem}
  \begin{proof}
    Obviously, there is an exact sequence $0\to\Om M\to\Om N\+P\to\Om L\to0$ for some projective module $P$. On the other hand, as
    $d_{\al\be}(N)=d_{\al\be}(M)+d_{\al\be}(L)$, the formulae \eqref{e25} and \eqref{e26} imply that $\rk\Om N=\rk\Om M+\rk L$. Hence
    $P=0$    and we obtain \eqref{e27}. By duality, we also have \eqref{e28}.
  \end{proof}
  
  \begin{cor}\label{c32} 
   Every exact sequence of regular $A$-lattices $0\to M\to N\to L\to0$ induces exact sequences of cohomologies
   \[
   0\to \hH^n(K,M)\to \hH^n(K,N)\to  \hH^n(K,L)\to  0.
   \]
  \end{cor}
  \begin{proof}
   If $n=0$, it follows from Proposition~\ref{p23}. For other $n$ it is obtained by an easy induction using Lemma~\ref{l31}.
  \end{proof}
   
   For regular lattices we can give an explicit form of cocycles defining cohomology classes. Namely,
   for an indecomposable regular lattice $M$ and an integer $n$ we set
  \[
   M(n)=
   \begin{cases}
   M_{pp} &\text{if $n$ is even},\\
   M_{0p} &\text{if $n$ is odd and $M\notin\kT^\8$},\\
   M_{p0} &\text{if $n$ is odd and $M\in\kT^\8$}.
   \end{cases}  
  \]
  For $n>0$ we define a homomorphism $M(n)\to\hH^n(K,M)$ mapping an element $a\in M(n)$ to the class of the cocycle $\xi_a:P_n\to M$
  defined as follows:
  \begin{itemize}
  \item[$\bu$]  If $M\notin \kT^\8$, then
  \[
   \xi_a(u^kv^{n-k})=
   \begin{cases}
    a &\text{if } k=n,\\
    0 &\text{otherwise}.
   \end{cases}
  \]
    \item[$\bu$]  If $M\in \kT^\8$, then
  \[
   \xi_a(u^kv^{n-k})=
   \begin{cases}
    a &\text{if } k=0,\\
    0 &\text{otherwise}.
   \end{cases}
  \]
  \end{itemize}
  
  \begin{thm}\label{t33} 
  The map $a\mapsto\xi_a$ induces an isomorphism $$\xi:M(n)/pM(n)\simeq \hH^n(K,M)$$
   for every $n>0$ and every regular indecomposable $A$-lattice $M$.
  \end{thm}
  \begin{proof}
  One easily sees that $\xi_a$ is a cocycle. Theorem~\ref{t26} and formulae \eqref{erk} show that $\hH^n(K,M)\simeq M(n)/pM(n)$. Hence we only
  have to prove that $\xi_a$ is not a coboundary if $a\notin pM(n)$. First we check it for the lattices $T^f_1$ and $T^{\la i}_1$. 
  
  We consider the case when $n$ is even and $M=T^f_1$, where $\deg f=d$ and $f\notin\{t,t-1\}$. 
  The other cases are quite similar or even easier.
  The corresponding representation of the quiver $\Gamma$  is
  \[
   \vcenter
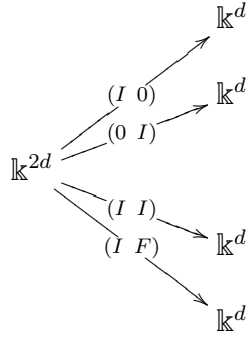
{ \xymatrix@R=1em{ && \Bbbk^d	 \\ && \Bbbk^d \\
   \Bbbk^{2d} \ar[rruu]|{(I\ 0)}	\ar[rru]|{(0\ I)}	\ar[rrd]|{(I\ I)}	\ar[rrdd]|{(I\ F)} \\
    && \Bbbk^d	 \\ && \Bbbk^d } }  
  \]
  where $I$ is $d\times d$ unit matrix and $F$ is the Frobenius matrix with the characteristic polynomial $f(t)$. Therefore, $M$ is the submodule of
  $\bop_{\al\be}M\sh_{\al\be}$, where $M\sh_{\al\be}=R_{\al\be}^d$ and $M$ consists of the quadruples 
  $a=(a_{pp},a_{p0},a_{0p},a_{00})\equiv(r,r',r+r',r+\tF r')\pmd$, where $r,r'\in R^d$ and $\tF$ is a
  $d\times d$ matrix over $R$ such that  $F=\tF\pmd$. Hence, $M_{\al\be}=pM\sh_{\al\be}$. In particular, elements $a\in M(n)$
  are of the form $(pr,0,0,0)$.  Let $\xi_a=d\ga$, where
  $\ga(x^{k-1}y^{n-k})=\ga_k\equiv(r_k,r'_k,r_k+r'_k,r_k+\tF r'_k)\pmd$ for $1\le k\le n$. Then
  \[
   d\ga(v^n) =0=y\ga_1\equiv(pr_1,0,p(r_1+r'_1),0)\ppmd,
  \]
  hence $\ga_1\equiv0\pmd$. Suppose that $\ga_{k-1}\equiv0\pmd$ for some $1<k\le n$. If $k$ is odd, then
  \begin{align*}
   d\ga(u^{k-1}v^{n-k+1})=0&=x\ga_{k-1}+y\ga_k\equiv\\
   	 &\equiv(pr_k,0,p(r_k+r'_k),0)\ppmd,
  \end{align*}
  If $k$ is even, then
    \begin{align*}
   d\ga(u^{k-1}v^{n-k+1})=0&=(x-p)\ga_{k-1}-(y-p)\ga_k\equiv\\
   	 &\equiv(0,pr'_k,0,p(r_k+\tF r'_k))\ppmd.
  \end{align*}
  In both cases $\ga_k\equiv0\pmd$. Therefore, $\ga_k\equiv0\pmd$ for all $1\le k\le n$. Then
  \[
    d\ga(u^n)=(a,0,0,0)=x\ga_n\equiv0\ppmd,
  \]
  so $a\in pM(n)$.  
 
  Suppose now that the theorem is valid for all $T^f_{k-1}$ and for all $T^{\la i}_{k-1}$. If $M=T^f_k$ or $M=T^{\la i}_k$,
  there is an exact sequence $0\to M'\to M\to M''\to 0$, where, respectively, $M'=T^f_1,\,M''=T^f_{k-1}$ or 
  $M'=T^{\la i},\,M''=T^{\la j}_{k-1}\ (j\ne i)$. It gives a commutative diagram with exact rows
  \[
      \xymatrix{ 0 \ar[r] & M'(n) \ar[d]^\xi	\ar[r] & M(n) \ar[d]^\xi \ar[r] & M''(n) \ar[d]^\xi \ar[r] & 0  \\
    0 \ar[r] & \hH^n(K,M') 	\ar[r] & \hH^n(K,M) \ar[r] & \hH^n(K,M'') \ar[r] & 0      
      }
  \]
  Using induction, we may suppose that the first and the third homomorphisms $\xi$ satisfy the theorem. Therefore, 
  so does the second, which accomplishes the proof.
  \end{proof}
  
  Dualizing this construction, we obtain an explicit description of Tate cohomologies with negative indices.
  Namely, for $n<0$ we define a homomorphism $M(n)\to\hH^n(K,M)$ mapping an element $a\in M(n)$ to the class of the cocycle $\hxi_a:P_n\to M$
  defined as follows:
  \begin{itemize}
  \item[$\bu$]  If $M\notin \kT^\8$, then
  \[
   \hxi_a(\hu^k\hv^{|n|-1-k})=
   \begin{cases}
    a &\text{if } k=|n|-1,\\
    0 &\text{otherwise}.
   \end{cases}
  \]
    \item[$\bu$]  If $M\in \kT^\8$, then
  \[
   \hxi_a(\hu^k\hv^{|n|-1-k})=
   \begin{cases}
    a &\text{if } k=0,\\
    0 &\text{otherwise}.
   \end{cases}
  \]
  \end{itemize}
  
  \begin{thm}\label{t34} 
  The map $a\mapsto\hxi_a$ induces an isomorphism $$\hxi:M(n)/pM(n)\simeq \hH^n(K,M)$$
   for every $n<0$ and every regular indecomposable $A$-lattice $M$.
  \end{thm}
  The proof just repeats that of Theorem~\ref{t33}, so we omit it.

\end{document}